\documentclass[11pt,reqno]{amsart}
\usepackage{amssymb}
\usepackage{upgreek}
\usepackage{dsfont }
\usepackage{mathrsfs}
\usepackage[all]{xy}
\newtheorem{thm}{Theorem}[section]
\newtheorem{lem}[thm]{Lemma}
\newtheorem{prop}[thm]{Proposition}

\theoremstyle{definition}

\newtheorem{ex}[thm]{Example}

\theoremstyle{remark}
\newtheorem{remark}[thm]{Remark}
\newtheorem{remarks}[thm]{Remarks}

\newtheorem{notation}[thm]{Notation}

\newcommand{\CA}{{\mathcal{A}}}

\newcommand{\bE}{{\overline{\mathcal{E}}}}

\newcommand{\CL}{{\mathcal{L}}}
\newcommand{\CB}{{\mathcal{B}}}
\newcommand{\CO}{{\mathcal{O}}}

\newcommand{\af}{\alpha}
\newcommand{\bt}{\beta}
\newcommand{\gm}{\gamma}
\newcommand{\dt}{\delta}

\newcommand{\ld}{\lambda}

\newcommand{\sm}{\sigma}

\newcommand{\om}{\omega}

\begin{document}


\title[Finite simple labeled graph 
 $C^*$-algebras  of Cantor minimal subshifts]
{Finite simple labeled graph $C^*$-algebras 
 of Cantor minimal subshifts}

\author[J. A. Jeong]{Ja A Jeong$^{\dagger}$}
\thanks{Research partially supported by 
             NRF-2012R1A1A2008160$^{\dagger}$ and NRF-2015R1C1A2A01052516$^{\dagger}$}
\thanks{Research partially supported by Hanshin University$^{\ddagger}$}
\thanks{Research partially supported by BK21 PLUS SNU Mathematical Sciences Division$^{*}$}
\address{
Department of Mathematical Sciences and Research Institute of Mathematics\\
Seoul National University\\
Seoul, 151--747\\
Korea} \email{jajeong\-@\-snu.\-ac.\-kr }

\author[E. J. Kang]{Eun Ji Kang$^{\dagger}$}
\address{
Department of Mathematical Sciences and Research Institute of Mathematics\\
Seoul National University\\
Seoul, 151--747\\
Korea} \email{kkang33\-@\-snu.\-ac.\-kr }

\author[S. H. Kim]{Sun Ho Kim$^{\dagger, *} $}
\address{
BK21 Plus Mathematical Sciences Division\\
Seoul National University\\
Seoul, 151--747\\
Korea} \email{sunho.kim.math\-@\-gmail.\-com }

\author[G. H. Park]{Gi Hyun Park$^{\ddagger}$}
\address{
Department of Financial Mathematics\\
Hanshin University\\
Osan, 447--791\\
Korea} \email{ghpark\-@\-hs.\-ac.\-kr }



\begin{abstract}   
It is now well known that a simple graph $C^*$-algebra 
$C^*(E)$ of a directed graph $E$ is either AF or purely infinite. 
In this paper, we address the question of whether this is the case 
for labeled graph $C^*$-algebras  recently introduced by Bates and Pask 
as one of the generalizations of graph $C^*$-algebras, and    
show that 
there exists a family of simple unital labeled graph $C^*$-algebras 
which are neither  AF nor purely infinite.
Actually these algebras are shown to be isomorphic to crossed products 
$C(X)\times_T \mathbb Z$ where  
the dynamical systems $(X,T)$ are Cantor minimal subshifts. 
Then it is an immediate consequence of well known results about 
this type of crossed products that each labeled graph $C^*$-algebra  
in the family obtained here is an $A\mathbb T$ algebra   
with real rank zero and has $\mathbb Z$ as its $K_1$-group.
\end{abstract}

\maketitle

\setcounter{equation}{0}

\section{Introduction}

\noindent

With the motivation to provide a common framework for studying 
the ultragraph $C^*$-algebras  
(\cite{To1, To2}) and 
the shift space $C^*$-algebras   
(see \cite{Ca, CM, Ma} among others), Bates and Pask \cite{BP1}
introduced the $C^*$-algebras associated to labeled graphs (more 
precisely, labeled spaces). 
Graph $C^*$-algebras 
(see \cite{BHRS, BPRS, KPR, KPRR, R} among many others) 
and Exel-Laca algebras \cite{EL} are 
ultragraph $C^*$-algebras and all these 
algebras are defined as universal objects 
generated by partial isometries 
and projections 
satisfying certain relations 
determined by graphs (for graph $C^*$-algebras), ultragraphs 
(for ultragraph $C^*$-algebras), 
and infinite matrices (for Exel-Laca algebras). 
In a similar but more complicated manner, 
a labeled graph $C^*$-algebra 
$C^*(E,\CL, \CB)$ is also defined as a   
$C^*$-algebra generated by partial isometries 
$\{s_a:a\in \CA\}$ and  projections $\{p_A:A\in \CB\}$, 
where $\CA$ is an alphabet 
onto which a {\it labeling map} $\CL:E^1\to \CA$ is 
given from the edge set $E^1$ of the directed graph $E$,  
and $\CB$, an {\it accommodating set}, 
is a set of vertex subsets $A\subset E^0$ satisfying certain conditions. 
The family of these generators is assumed to obey a set of rules 
regulated by the triple $(E,\CL,\CB)$ called a {\it labeled space} 
and moreover it should be universal in the sense 
that any $C^*$-algebra generated by 
a family of partial isometries and 
projections satisfying the same rules must be a quotient 
algebra of 
$C^*(E, \CL,\CB)$. 
The universal property allows the group $\mathbb T$ 
to act on $C^*(E,\CL,\CB)$ in a canonical way, and 
this action $\gm$ (called the {\it gauge action})
plays an important role throughout the study of generalizations of 
the Cuntz-Krieger algebras. 
The Cuntz-Krieger algebras \cite{CK} 
(and the Cuntz algebras \cite{C81})
are the $C^*$-algebras of finite graphs from which 
many generalizations have emerged in various ways including 
the $C^*$-algebras of higher-rank graphs whose study started in 
\cite{KP}.

Simplicity and pure infiniteness results for 
labeled graph $C^*$-algebras are obtained in \cite{BP2}, 
and particularly 
it is shown that there exists a purely infinite simple 
labeled graph $C^*$-algebra which is not 
stably isomorphic to any graph $C^*$-algebras. 
Thus we can say that the class of simple labeled graph $C^*$-algebras 
is strictly larger than that of simple graph $C^*$-algebras. 
As is shown in \cite{To2}, every simple ultragraph $C^*$-algebra 
is either AF or purely infinite, whereas      
we know from \cite{PRRS} that  
among higher rank graph $C^*$-algebras   
there exist simple $C^*$-algebras which are  
neither AF nor purely infinite, 
more specifically there exist such simple $C^*$-algebras which are  
stably isomorphic to irrational rotation algebras or Bunce-Deddens algebras. 
These examples of finite (but non-AF) simple $C^*$-algebras 
associated to higher rank graphs raise a natural question 
of whether 
there exist labeled graph $C^*$-algebras that 
are simple finite but non-AF.   
The purpose of this paper is to answer this question positively 
by providing a family of such simple  
labeled graph $C^*$-algebras. 
The  $C^*$-algebras in this family  are 
$A\mathbb T$-algebras (limit circle algebras) with traces  
that are isomorphic to 
crossed products $C(X)\times_T \mathbb Z$ 
of Cantor minimal systems $(X,T)$,  
where the compact metric spaces $X$ are subshifts 
over finite alphabets.

 A dynamical system $(X,T)$ consists of 
 a compact metrizable space $X$ and a 
 transformation $T:X\to X$ which is a homeomorphism. 
 This determines a $C^*$-dynamical system 
 $(C(X), \mathbb Z, T)$ where 
 $T(f):=f\circ T^{-1}$, $f\in C(X)$ and thus 
 gives rise to the crossed product 
 $C(X)\times_T \mathbb Z$.   
 If two dynamical systems $(X_i,T_i)$, $i=1,2$, are topologically 
 conjugate, namely if there is an homeomorphism $\phi: X_1\to X_2$ 
 satisfying  
 $T_2(\phi(x))=\phi(T_1(x))$ for all $x\in X$, then 
 it is rather obvious that the crossed products are 
 isomorphic. 
As a consequence of the Markov-Kakutani fixed point theorem, 
one can show that there exists a Borel probability measure $m$ on $X$ 
which is $T$-invariant in the sense that 
$m\circ T^{-1}=m$ (for example, see \cite[Theorem VIII. 3.1]{Da}). 
If there exists a unique $T$-invariant measure, we call 
$(X,T)$ {\it uniquely ergodic}.   
 If $X$ is the only non-empty closed $T$-invariant subspace of $X$, the system 
 $(X,T)$ is said to be {\it minimal}, and  
 as is well known, a dynamical system $(X,T)$ is minimal 
 if and only if  
 each $T$-orbit $\{T^ix: i\in \mathbb Z\}$, $x\in X$, 
 is dense in $X$. 
 A  Cantor space  is characterized as a compact metrizable 
 totally disconnected space with no isolated points,  
 and a dynamical system $(X,T)$ on a Cantor space $X$ 
 is called a {\it Cantor system}.
The family of Cantor minimal systems is important for the study  
of whole minimal dynamical systems in view of the fact that 
every minimal system is a factor of a Cantor minimal system  
(see \cite[Section 1]{GPS}). 

For a dynamical system $(X,T)$ on an infinite space $X$,  
the crossed product $C(X)\times_{T} \mathbb Z$  
is well known to be simple exactly when 
the system $(X,T)$ is minimal. 
In particular, if $(X,T)$ is a minimal dynamical system on a Cantor space 
$X$, this simple crossed product turns out to be an $A\mathbb T$-algebra, 
an inductive limit of finite direct sums of 
matrix algebras over $\mathbb C$ or  $C(\mathbb T)$ 
 (for example, see \cite[Chapter VIII]{Da}). 
It should be noted here 
that these simple crossed products 
$C(X)\times_{T} \mathbb Z$ of Cantor minimal systems are never
AF since their $K_1$ groups are 
all equal to $\mathbb Z$, hence nonzero (\cite[Theorem 1.4]{HPS}). 

For a finite alphabet $\CA$ ($|\CA |\geq 2$), 
the set $\CA^{\mathbb Z}$ of all two-sided infinite sequences 
becomes a compact metrizable space in the product topology 
and forms a dynamical system $(\CA^{\mathbb Z}, T)$ 
together with the shift transformation $T$ given by 
$T(\om)_i:=\om_{i+1}$, 
$\om\in \CA^{\mathbb Z}$, $i\in \mathbb Z$. 
If $X\subset \CA^{\mathbb Z}$ is a $T$-invariant closed subspace, 
we call the dynamical system $(X,T)$ a {\it subshift} of $(\CA^{\mathbb Z}, T)$. 
For  a sequence $\om\in \CA^{\mathbb Z}$, let $\CO_\om$ denote 
the the closure of the $T$-orbit of $\om$. 
Then, as is well known, the subshift $(\CO_\om, T)$ 
becomes a Cantor minimal system 
whenever $\om$ is an almost periodic and aperiodic sequence. 

In order to explain how we form a labeled space 
from a  Cantor minimal subshift $(\CO_\om, T)$, 
let $E_{\mathbb Z}$ be the directed graph 
with the vertex set $\{v_n:n\in \mathbb Z\}$ and 
the edge set $\{ e_n: n\in \mathbb Z\}$, where 
each $e_n$ is an arrow from $v_n$ to $v_{n+1}$, $n\in \mathbb Z$. 
Then we consider a labeling map 
$\CL_\om$ on the graph $E_{\mathbb Z}$ which assigns 
to an edge $e_n$ a letter $\om_n$  for each $n\in \mathbb Z$. 
In this way we obtain a labeled graph $C^*$-algebra 
$C^*(E_{\mathbb Z}, \CL_\om, \bE_{\mathbb Z})$, where 
$\bE_{\mathbb Z}$ is the smallest set  amongst the normal 
accommodating sets. 
Then we first show that these unital labeled graph algebras 
$C^*(E_{\mathbb Z}, \CL_\om, \bE_{\mathbb Z})$ are all simple 
and have traces. 
In the simple crossed product $C(\CO_\om)\times_T \mathbb Z$, 
we then find a family of partial isometries   
and projections satisfying the same relations 
required for the canonical generators of 
$C^*(E_{\mathbb Z}, \CL_\om, \bE_{\mathbb Z})$, 
which proves from universal property of 
$C^*(E_{\mathbb Z}, \CL_\om, \bE_{\mathbb Z})$ 
that there exists an isomorphism  
of $C^*(E_{\mathbb Z}, \CL_\om, \bE_{\mathbb Z})$  to 
the crossed product $C(\CO_\om)\times_T \mathbb Z$. 
Our results can be summarized as follows: 
\vskip 1pc 

\begin{thm} (Theorem~\ref{main thm 1} and Theorem~\ref{main thm 2})
Let $\CA$ be a finite alphabet with $|\CA|\geq 2$, 
and let $\om\in \CA^{\mathbb Z}$ be a sequence such that 
the subshift $(\CO_\om,T)$ is a Cantor minimal  system. 
If $\CL_\om$ is a labeling map on the graph $E_{\mathbb Z}$ 
by the sequence $\om$,  the labeled graph $C^*$-algebra 
$C^*(E_{\mathbb Z}, \CL_\om,\bE_{\mathbb Z})$ is a non-AF simple unital  $C^*$-algebra. 
Moreover there is an isomorphism   
$$\pi: C^*(E_{\mathbb Z}, \CL_\om,\bE_{\mathbb Z})\to 
C(\CO_\om) \times_{T} \mathbb Z$$ 
such that the restriction of $\pi$ onto 
the fixed point algebra $C^*(E_{\mathbb Z}, \CL_\om,\bE_{\mathbb Z})^\gm$
of the gauge action $\gm$ is an isomorphism onto 
$C(\CO_\om)$.
\end{thm} 

\vskip 1pc\noindent  
The crossed products $C(X)\times_T \mathbb Z$ of 
Cantor minimal systems have been studied intensively 
(especially in \cite{GPS, HPS}). 
Perhaps one important result from the works, in our viewpoint, 
would be the fact that the crossed products
$C(\CO_\om) \times_{T} \mathbb Z$ can be completely classified 
by their ordered $K_0$-groups with distinguished oder units 
(\cite[Theorem 2.1]{GPS}). 
Also from the above theorem  and \cite[Theorem 1.4]{HPS}  
we know  that 
the labeled graph $C^*$-algebras $C^*(E_{\mathbb Z}, \CL_\om,\bE_{\mathbb Z})$ 
are  $A\mathbb T$-algebras  with 
$K_1(C^*(E_{\mathbb Z}, \CL_\om,\bE_{\mathbb Z}))=\mathbb Z$, hence 
they are not AF.

Finally, 
regarding the question of abundance of those Cantor minimal subshift systems,  
we notice a well known fact that 
$(X,T)$ is topologically conjugate to a two-sided subshift 
if and only if it is  expansive, and also from  \cite[Theorem 1]{DM} that 
this is the case if a Cantor system $(X,T)$ 
has a finite rank $K$ and $K\geq 2$  while  
odometer systems are the systems of rank one 
(we refer the reader to \cite{DM} for definitions and 
properties of this sort of systems).

\vskip 1pc 

\section{Preliminaries}

\subsection{\bf Labeled spaces} 
 
We will follow notational conventions of \cite{KPR} for 
graph $C^*$-algebras and of \cite{BP2, BCP} for
labeled spaces and their $C^*$-algebras. 
A  {\it directed graph}  $E=(E^0,E^1,r,s)$
consists of a countable vertex set $E^0$, a countable edge set $E^1$,
and the range, source maps $r$, $s: E^1\to E^0$. 
If $v\in E^0$ emits (receives, respectively) 
no edges it is called a {\it sink} ({\it source}, respectively). 
Throughout this paper, we assume that 
{\it  graphs have no sinks and no sources}.
 
$E^n$ denotes the set of all finite paths $\ld=\ld_1\cdots \ld_n$
of {\it length} $n$ ($|\ld|=n$),
($\ld_{i}\in E^1,\ r(\ld_{i})=s(\ld_{i+1}), 1\leq i\leq n-1$).
We write $E^{\leq n}$ and $E^{\geq n}$  
for the sets $\cup_{i=1}^n E^i$ and $\cup_{i=n}^\infty E^i$, respectively.
The range and source maps, $r$ and $s$, naturally extend to 
all finite paths $E^{\geq 0}$, where 
$r(v)=s(v)=v$ for $v\in E^0$. 
If a sequence of edges $\ld_i\in E^1(i\geq 1)$  satisfies  $r(\ld_{i})=s(\ld_{i+1})$,
one has an infinite path $\ld_1\ld_2\ld_3\cdots $
with the source vertex $s(\ld_1\ld_2\ld_3\cdots):=s(\ld_1)$.  
By $E^\infty$ we denote the set of all infinite paths.

 A {\it labeled graph} $(E,\CL)$ over a countable alphabet  $\CA$
 consists of a directed graph $E$ and
 a {\it labeling map} $\CL:E^1\to \CA$. 
 For $\ld=\ld_1\cdots\ld_n\in E^{\geq 1}$, we call 
 $\CL(\ld):=\CL(\ld_1)\cdots\CL(\ld_n)$ a ({\it labeled}) {\it path}.
 Similarly one can define an infinite labeled path 
 $\CL(\ld)$ for $\ld\in E^\infty$.  
A labeled graph $(E,\CL)$ is said to have a {\it repeatable path} $\bt$ 
if $\bt^n:=\bt\cdots\bt(\text{repeated $n$-times})\in\CL(E^{\geq 1})$ 
for all $n\geq 1$. 
 The {\it range} $r(\af)$ of a labeled path 
 $\af\in \CL(E^{\geq 1})$ is defined to be a vertex subset of $E^0$:
$$
r(\af)  =\{r(\ld) \,:\, \ld\in E^{\geq 1},\,\CL(\ld)=\af\},
$$
and  the {\it source} $s(\af)$ of $\af$ is defined similarly.
 The {\it relative range of $\af\in \CL(E^{\geq 1})$
 with respect to $A\subset 2^{E^0}$} is defined to be
$$
 r(A,\af)=\{r(\ld)\,:\, \ld\in E^{\geq 1},\ \CL(\ld)=\af,\ s(\ld)\in A\}.
$$ 
For notational convenience, we use a symbol  $\epsilon$ such that 
$r(\epsilon) =E^0$, $r(A, \epsilon) = A$ for all $A \subset E^0$,
and $\af=\epsilon\af=\af\epsilon$ for all $\af\in \CL(E^{\geq 1})$, and 
write 
$$\CL^\#(E):=\CL(E^{\geq 1})\cup \{\epsilon \}.$$  
We denote the subpath $\af_i\cdots \af_j$ of 
$\af=\af_1\af_2\cdots\af_{|\af|}\in \CL(E^{\geq 1})$
by $\af_{[i,j]}$ for $1\leq i\leq j\leq |\af|$.
A subpath of the form $\af_{[1,j]}$ is called
an {\it initial path} of $\af$. 
The symbol $\epsilon$ is regarded as an initial (and terminal)  
path of every path. 
 
 Let $\CB\subset 2^{E^0}$ be a collection of subsets of $E^0$.  
 If $r(A,\af)\in \CB$ for all $A\in \CB$ and $\af\in \CL(E^{\geq 1})$, 
 $\CB$ is said to be
 {\it closed under relative ranges} for $(E,\CL)$.
 We call $\CB$ an {\it accommodating set} for $(E,\CL)$
 if it is closed under relative ranges,
 finite intersections and unions and contains $r(\af)$ for all $\af\in \CL(E^{\geq 1})$.
The triple $(E,\CL,\CB)$ is called  a {\it labeled space} when 
$\CB$ is accommodating for $(E,\CL)$.

 For $A,B\in 2^{E^0}$ and $n\geq 1$, let
 $$ AE^n =\{\ld\in E^n\,:\, s(\ld)\in A\},\ \
  E^nB=\{\ld\in E^n\,:\, r(\ld)\in B\}.$$
 We write $E^n v$ for $E^n\{v\}$ and $vE^n$ for $\{v\}E^n$,
 and will use notations like $AE^{\geq k}$ and $vE^\infty$
 which should have obvious meaning.
 A labeled space $(E,\CL,\CB)$ is said to be {\it set-finite}
 ({\it receiver set-finite}, respectively) if for every $A\in \CB$ and $l\geq 1$ 
 the set  $\CL(AE^l)$ ($\CL(E^lA)$, respectively) is finite. 
 A labeled space $(E,\CL,\CB)$ is {\it finite} if 
 there are only finitely many labels.  

In this paper, we will always assume that 
labeled spaces $(E,\CL,\CB)$ are 
{\it weakly left-resolving}, namely  
$$r(A,\af)\cap r(B,\af)=r(A\cap B,\af)$$
for all $A,B\in \CB$ and  $\af\in \CL(E^{\geq 1})$. 
  $(E,\CL,\CB)$ is {\it left-resolving} if 
  $\CL : r^{-1}(v) \rightarrow \mathbf{\CA}$ is injective 
  for each $v \in E^0$.
Left-resolving labeled spaces are weakly left-resolving.

For each $l\geq 1$, the following relation on $E^0$, 
$$  v\sim_l w \ \text{ if and only if \ }  \CL(E^{\leq l} v)=\CL(E^{\leq l} w)$$
is actually an equivalence relation, and 
the equivalence class $[v]_l$ of $v\in E^0$  is called a 
{\it generalized vertex}.
   If $k>l$,  $[v]_k\subseteq [v]_l$ is obvious and
   $[v]_l=\cup_{i=1}^m [v_i]_{l+1}$
   for some vertices  $v_1, \dots, v_m\in [v]_l$ (\cite[Proposition 2.4]{BP2}).

\vskip 1pc

\begin{notation} 
Given a  labeled graph $(E,\CL)$,  
$\bE$  denotes the smallest {\it normal} accommodating set, that is 
the smallest one among the accommodating sets which are 
closed under relative complements.
\end{notation}

\vskip 1pc

\begin{prop}{\rm (\cite[Remark 2.1 and Proposition 2.4.(ii)]{BP2},  
 \cite[Proposition 2.3]{JKK})}
\label{prop-barE} 
Let $(E,\CL)$ be a labeled graph ($E$ has no sinks or sources). Then 
$$\bE=\{ \cup_{i=1}^{n}  [v_i]_l \,:\,  
v_i  \in E^0,\ l,n\geq 1 \}.$$
\end{prop} 

\vskip 1pc

\subsection{\bf Labeled graph $C^*$-algebras} 
Here we review the labeled graph $C^*$-algebras 
which are associated to set-finite, receiver set-finite, 
and weakly left-resolving labeled spaces 
(whose underlying graphs have no sinks or sources) 
although our results are concerning only about 
finite left-resolving spaces.  

Let $(E,\CL,\CB)$ be a  labeled space such that
$\bE\subset \CB$.
Recall from \cite[Definition 2.1]{BCP} that 
a {\it representation} of $(E,\CL,\CB)$ is a collection of 
projections $\{p_A : A\in \CB\}$ and
partial isometries
$\{s_a : a\in \CA\}$ such that for $A, B\in \CB$ and $a, b\in \CA$,
\begin{enumerate}
\item[(i)]  $p_{\emptyset}=0$, $p_Ap_B=p_{A\cap B}$, and
$p_{A\cup B}=p_A+p_B-p_{A\cap B}$,
\item[(ii)] $p_A s_a=s_a p_{r(A,a)}$,
\item[(iii)] $s_a^*s_a=p_{r(a)}$ and $s_a^* s_b=0$ unless $a=b$,
\item[(iv)] for each $A\in \CB$,
\begin{eqnarray}\label{representation} p_A=\sum_{a\in \CL(AE^1)} s_a p_{r(A,a)}s_a^*.
\end{eqnarray}
\end{enumerate}
It follows from  (iv) that $p_A=\sum_{\af\in\CL(AE^n)}s_\af  p_{r(A,\af)}s_\af^*$ 
for $n\geq 1$.
 By $C^*(p_A,s_a)$ we denote the $C^*$-algebra generated by 
$\{s_a,p_A: a\in \CA,\, A\in \CB \}$.

\vskip 1pc

\begin{remark}\label{review remarks}
Let $(E,\CL,\CB)$ be a  labeled space such that
$\bE\subset \CB$.
\begin{enumerate}
\item[(i)] 
There exists a $C^*$-algebra 
generated by a universal representation
$\{s_a,p_A\}$ of $(E,\CL,\CB)$ (see the proof of \cite[Theorem 4.5]{BP1}). 
If $\{s_a,p_A\}$ is a universal representation of $(E,\CL,\CB)$,
we call $C^*(s_a,p_A)$, denoted $C^*(E,\CL,\CB)$, the {\it labeled graph $C^*$-algebra} of
$(E,\CL,\CB)$. 
 Note that   $s_a\neq 0$ and $p_A\neq 0$ for $a\in \CA$
and  $A\in \CB$, $A\neq \emptyset$, and that
$s_\af p_A s_\bt^*\neq 0$ if and only if $A\cap r(\af)\cap r(\bt)\neq \emptyset$.
By definition of representation
and \cite[Lemma 4.4]{BP1},
it follows that  
\begin{eqnarray}\label{eqn-elements}
C^*(E,\CL,\CB)=\overline{span}\{s_\af p_A s_\bt^*\,:\,
\af,\,\bt\in  \CL^{\#}(E),\ A\in \CB\},
\end{eqnarray}
where $s_\epsilon$ is regarded as the unit of the multiplier algebra of 
$C^*(E,\CL,\mathcal B)$. 

\item[(ii)] Universal property of  $C^*(E,\CL,\CB)=C^*(s_a, p_A)$ 
defines the {\it gauge action} 
$\gm:\mathbb T\to Aut(C^*(E,\CL,\CB))$ such  that 
for $a\in \CL(E^1)$, $A\in \CB$, and $z\in \mathbb T$, 
$$\gm_z(s_a)=z s_a \ \text{ and } \  \gm_z(p_A)=p_A.$$

\item[(iii)]  
The fixed point algebra of $\gm$ is an AF algebra such that 
\begin{eqnarray}\label{fixed point}
C^*(E,\CL,\CB)^\gm=\overline{\rm span}\{s_\af p_A s_\bt^*: |\af|=|\bt|,\ A\in \CB\}
\end{eqnarray}
Moreover, since $\mathbb T$ is a compact group, 
there exists a faithful conditional expectation 
$$\Psi: C^*(E,\CL,\CB)\to C^*(E,\CL,\CB)^\gm.$$

\end{enumerate} 
\end{remark} 

\vskip 1pc 

Recall \cite{BP2, JK} that for a labeled space $(E,\CL,\bE)$, 
a path $\af\in \CL([v]_l E^{\geq 1})$ is {\it agreeable} for 
a generalized vertex $[v]_l$ if 
$\af=\bt^k\bt'$ for some $\bt\in \CL([v]_l E^{\leq l})$ and its initial path $\bt'$, and $k\geq 1$. 
A labeled space $(E,\CL,\bE)$  is said to be {\it disagreeable} 
if every $[v]_l$, $l\geq 1$, $v\in E^0$, is disagreeable 
in the sense that there is an $N\geq 1$ such that for all $n\geq N$ 
there is a path $\af\in \CL([v]_l E^{\geq n})$  which is not {\it agreeable}. 

\vskip 1pc 

\begin{remark}\label{repeatable} 
If $(E,\CL,\bE)$  is disagreeable, every representation $\{s_a,p_A\}$ such that 
$p_A\neq 0$ for all non-empty set $A\in \bE$ gives rise to a $C^*$-algebra 
$C^*(s_a,p_A)$ isomorphic to $C^*(E,\CL,\bE)$ (\cite[Theorem 5.5]{BP2} and 
\cite[Corollary 2.5]{JKP}). 
A labeled space $(E,\CL,\bE)$ is disagreeable if 
there is no repeatable paths in $(E, \CL)$ (\cite[Proposition 4.12]{JKK}).
\end{remark}
  
\vskip 1pc 
 
 $K$-theory of labeled graph $C^*$-algebras was obtained 
 in \cite{BCP}. 
 Let $(E,\CL,\CB)$ be a normal labeled space. 
 Since we assume that $E$ has no sink vertices ($E^0_{\rm sink} =\emptyset$), 
 the set  $\CB_J$ given in (2) of \cite{BCP}  
 coincides with $\CB$, and by \cite[Theorem 4.4]{BCP} 
 the linear map 
 $(1-\Phi) :{\rm span}_{\mathbb Z}\{\chi_A: A\in \CB\} \to 
  {\rm span}_{\mathbb Z}\{\chi_A: A\in\CB\}$ given by 
\begin{equation}\label{Phi}
(1-\Phi)(\chi_A)=\chi_A-\sum_{a\in \CL(AE^1)} \chi_{r(A,a)},\ \ A\in \CB
\end{equation} 
 determines the $K$-groups  of $C^*(E,\CL,\CB)$ as follows:
\begin{align} 
K_0(C^*(E,\CL,\CB))& \cong  {\rm span}_{\mathbb Z}\{\chi_A : A\in \CB \}/{\rm Im}(1-\Phi)\label{K0}\\ 
 K_1(C^*(E,\CL,\CB))& \cong \ker(1-\Phi)\label{K1}. 
\end{align} 
In (\ref{K0}), the isomorphism is given by $[p_A]_0\mapsto \chi_A+{\rm Im}(1-\Phi)$ for 
$A\in \CB$.

\vskip 1pc
 
\subsection{\bf  Cantor minimal systems that are subshifts} 
A (topological) dynamical system $(X,T)$ consists of 
a compact metrizable space $X$ 
and a homeomorphism $T$ on $X$.
By Krylov-Bogolyubov Theorem, a dynamical system $(X,T)$ 
admits a Borel probability measure $m$ which is $T$-invariant, 
that is $m(T^{-1}(E))=m(E)$ for all Borel sets $E$.
If there exists exactly one $T$-invariant probability measure, 
we say that the system $(X,T)$ is {\it uniquely ergodic}.
We will focus on the Cantor systems $(X,T)$ that are 
subshifts, and here we briefly review definitions and 
basic properties of such Cantor systems. 
 
For an alphabet $\CA$ ($|\CA|\geq 2)$, 
 a {\it word} (or {\it block}) over  $\CA$ 
is a  finite sequence $b=b_1\cdots b_k$ of symbols (or letters) $b_i$'s in $\CA$ 
of length $|b|:=k\geq 1$. 
By $\CA^+$, we denote the set of all {\it words}. 
Let  $\epsilon$ be  the empty word of length zero and let
$\CA^*:=\CA^+\cup \{\epsilon\}$.   
The set 
$$\CA^{\mathbb Z}:=\{\om=\cdots \om_{-1}\om_0\om_1\cdots : \om_i\in \CA\}$$ 
of all two-sided infinite sequences on $\CA$,  
endowed with the product topology of the discrete topology on  
$\CA$,  is a totally disconnected compact metrizable space. Actually  
the   {\it cylinder sets} 
$${}_t[b]:=\{\om \in \CA^{\mathbb Z} : \om_{[t, t+|b|-1]}=b\},$$
 $b\in \CA^+$, $t\in \mathbb Z$,  
are clopen and form a base for the topology, where 
$\om_{[t_1,t_2]}$ denotes the block $\om_{t_1}\cdots \om_{t_2}$  ($t_1\leq t_2$).
Thus the characteristic functions $\chi_{{}_t[b]}$  
are continuous  for all $b\in \CA^+$, $t\in \mathbb Z$. 
If $b=\om_{[t_1,t_2]}$ holds for $b\in \CA^+$ and 
$\om\in \CA^{\mathbb Z}\cup \CA^+$, 
$b$ is called a {\it factor} of $\om$.  
For $\om\in \CA^{\mathbb Z}$ (or  $\CA^{\mathbb N}$), 
the set of all factors of $\om$ is denoted by
$$L_\om=\{\om_{[t_1,t_2]} : t_1\leq t_2\}.$$
For convenience, we will  use the following notation:
$$[.b]:={}_0[b],\ \ [b.]:={}_{-|b|}[b],\ \  [b.c]:={}_{-|b|}[bc]$$ 
for words $b,c\in \CA^+$.

 The {\it shift} transform $T:\CA^{\mathbb Z}\to \CA^{\mathbb Z}$ given by 
 $$(Tx)_k=x_{k+1}, \  k\in \mathbb Z,$$
 is a homeomorphism. 
A {\it subshift} on $\CA$ is a (topological) dynamical system  $(X,T)$ 
which consists of a $T$-invariant closed subset $X\subset \CA^{\mathbb Z}$ and 
 the restriction   $T|_X$ which we denote by $T$ again.
 If we consider the shift transform $T$ on the space $\CA^{\mathbb N}$ of one-sided 
 infinite sequences, it is a
 continuous transform (but not a homeomorphism).

For $\om\in\CA^{\mathbb Z}$, the closure of the  orbit of $\om$  
is denoted by 
$$\CO_\om:=\overline{\{ T^i(\om): i\in \mathbb Z\}}\ \subset \CA^{\mathbb Z}.$$ 
A dynamical system $(X, T)$ is {\it minimal} if 
every orbit is dense in $X$, namely 
$\CO_x=X$ for all $x\in X$. 
It is well known  that a subshift $(\CO_\om,T)$ 
is minimal if and only if $\om$ is {\it almost periodic} 
(or {\it uniformly recurrent}) in the sense that 
each factor of $\om$ occurs with bounded gaps.

\vskip 1pc 
We provide examples of subshifts that are Cantor minimal systems: 

\vskip 1pc

\begin{ex}(\textbf{Generalized-Morse sequences})
(\cite{Ke}) \label{generalized Morse} 
Let $\CA=\{0,1\}$.  
For a one-sided sequence $x\in \CA^{\mathbb N}$, let
$ \mathscr{O}_x:=\{\om\in \CA^{\mathbb Z} : L_\om\subset L_x\}$.
Note that each block $b\in \CA^+$ 
defines a block $\tilde{b}$, 
called the {\it mirror image} of $b$, 
such that $\tilde{b}_i=b_i+1$ (mod 2). 
For $c=c_0\cdots c_n\in \CA^+$,  the product 
$b\times c$ of $b$ and $c$ denotes the block
(of length $|b|\times |c|$) obtained by  
putting $n+1$ copies of either $b$ or $\tilde b$ next to each other 
according to the rule of choosing the $i$th copy as $b$ if $c_i=0$ and 
$\tilde b$ if $c_i=1$. For example, if 
$b=01$ and $c=011$, then the product block $b\times c$ is equal to 
$b\tilde b \tilde b=011010$.

Let $\{b^i:=b^i_0\cdots b^i_{|b^i|-1}\}_{i\geq 1}\subset \CA^+$ 
be a sequence of blocks  with length $|b^i|\geq 2$ such that  
$b^i_0=0$ for all $i\geq 0$. 
Since the product operation $\times$ is associative, one can consider 
a sequence of the form 
$$x=b^0\times b^1\times b^2\times \cdots \in \CA^{\mathbb N}$$ 
which is called a (one-sided) {\it recurrent} sequence  
(see \cite[Definition 7]{Ke}). 
We call $x=b^0\times b^1\times b^2\times \cdots\in \CA^{\mathbb N}$ a 
{\it (generalized) one-sided  Morse sequence} if it is non-periodic and 
$$ \sum_{i=0}^\infty \min (r_0(b^i),r_1(b^i))=\infty,$$ 
where $r_a(b)$ is the {\it relative frequency of occurrence} of 
$a\in \CA$ in $b\in \CA^+$ 
(see \cite[p.338]{Ke}).
If  $x \in \CA^{\mathbb N}$ is a non-periodic recurrent sequence, 
it is almost periodic, and  
there exists $\om\in \mathscr{O}_x$ with $x= \om_{[0,\infty)}$.
Moreover, $x$ is a one-sided Morse sequence 
if and only if $\CO_\om$ is minimal and  uniquely ergodic, 
and if this is the case, then 
$\CO_\om=\mathscr{O}_x$. 

  By a {\it generalized Morse sequence}, 
 we mean a two-sided sequence $\om\in \CA^{\mathbb Z}$ such that 
 $x:=\om_{[0,\infty)}$ is a one-sided Morse sequence and 
 $L_\om=L_x$. 
 (Note that the term a {\it two-sided generalized 
 Morse sequence} used in \cite{Ke} means a sequence  
 $\om\in \mathscr{O}_x$ for some one-sided Morse sequence $x$.)    
 
The subshifts $(\CO_\om,T)$ for generalized Morse sequences 
$\om$ are uniquely ergodic Cantor minimal systems. 
\end{ex}

\vskip 1pc 
 
\begin{ex} (\textbf{Substitution subshifts}) (\cite{Ho}) 
Let $\CA$ be a finite alphabet with $|\CA|\geq 2$. 
A  {\it substitution} on $\CA$ is a map $\sm: \CA\to \CA^+$.   
   $\sm$ can be iterated to define maps $\sm^k: \CA\to \CA^+$ 
   for all positive integer $k$, 
   and is called {\it primitive} if there exists $k\geq 1$ such that 
$b$ appears in $\sm^k(a)$ for all $a,b\in \CA$.  
By the {\it language} $L_\sm$ of a substitution $\sm$ we mean 
the set of words that are factors of $\sm^k(a)$ 
for some $k\geq 1$ and $a\in \CA$. 
  The subshift              
$$X_\sm:=\{x\in \CA^{\mathbb Z}\mid L_x\subset L_\sm\},$$
 associated to this language $L_\sm$ is called 
 the {\it substitution subshift} defined by $\sm$. 
If $\sm$ is primitive, it is known that the system $(X_\sm,T)$  is minimal 
and thus a Cantor minimal system. 

A sequence $\om\in \CA^{\mathbb Z}$ is called a 
{\it fixed point} of $\sm$ if 
$\sm(\om)=\om$, and it is known that for any primitive substitution $\sm$, 
there is an $n\geq 1$ such that $\sm^n$ admits a fixed point $\om$ in $X_\sm$. 
Since $\sm^n$ and $\sm$ define the same dynamical system, we can 
only consider primitive substitutions $\sm$ with a fixed point $\om\in X_\sm$, 
and in this case,  $X_\sm=\CO_\om$ follows. 
To avoid the case where $X_\sm$ is finite, or equivalently $\om$ is 
shift periodic, we also assume that $\sm$ is an {\it aperiodic} substitution 
(giving rise to the infinite system $X_\sm$). 
Then the substitution subshifts $(X_\sm, T)=(\CO_\om,T)$ are 
uniquely ergodic minimal Cantor systems. 
\end{ex}

\vskip 1pc 
 
\begin{ex}(\textbf{Thue-Morse sequence})\label{Thue Morse} 
Let $\CA=\{0,1\}$ and $b^i:=01\in \CA^+$ for all $i\geq 0$. 
Then the recurrent sequence 
$$x:=b^0\times b^1\times b^2\times \cdots\ =\, 01\times b^1\times\cdots 
=0110\times b^2\times \cdots= 01101001\times b^3\times\cdots$$ 
is a one-sided Morse sequence called the {\it Thue-Morse sequence} and 
$$\om:=x^{-1}.x
=\cdots 10010110.011010011001 \cdots\in \mathscr{O}_x$$ 
is a generalized Morse sequence, 
where
$x^{-1}:=  \cdots x_{2} x_{1} x_{0}$ is the sequence obtained 
by writing $x= x_0 x_1 \cdots$ in reverse order.   
In fact, $\om$ is the sequence constructed from $x$ in the proof of 
\cite[Lemma 4]{Ke}, and 
it is well known \cite{GH} that  $\om$ is characterized as a sequence with 
no blocks of the form 
$b b b_0$ for any $b=b_0\cdots b_{|b|-1}\in \CA^+$. 
By Example~\ref{generalized Morse}, the subshift 
$(\CO_\om, T)$ is a uniquely ergodic Cantor minimal system.

On the other hand, this Thue Morse sequence $\om$ is 
the fixed point of the primitive 
aperiodic substitution $\sm:\CA\to \CA^+$ given by 
$$\sm(0)=01 \ \text{ and }\ \sm(1)=10,$$
so that the subshift $(\CO_\om,T)$ can also be viewed 
as a substitution subshift $(X_\sm,T)$.
\end{ex}

\vskip 1pc

\section{Main results}

\noindent
Throughout this section, 
$E_{\mathbb Z}$ will denote the following graph: 

\vskip 1.5pc 
{} \hskip 1.1cm
\xy  /r0.3pc/:(-44.2,0)*+{\cdots};(44.3,0)*+{\cdots .};
(-40,0)*+{\bullet}="V-4";
(-30,0)*+{\bullet}="V-3";
(-20,0)*+{\bullet}="V-2";
(-10,0)*+{\bullet}="V-1"; (0,0)*+{\bullet}="V0";
(10,0)*+{\bullet}="V1"; (20,0)*+{\bullet}="V2";
(30,0)*+{\bullet}="V3";
(40,0)*+{\bullet}="V4";
 "V-4";"V-3"**\crv{(-40,0)&(-30,0)};
 ?>*\dir{>}\POS?(.5)*+!D{};
 "V-3";"V-2"**\crv{(-30,0)&(-20,0)};
 ?>*\dir{>}\POS?(.5)*+!D{};
 "V-2";"V-1"**\crv{(-20,0)&(-10,0)};
 ?>*\dir{>}\POS?(.5)*+!D{};
 "V-1";"V0"**\crv{(-10,0)&(0,0)};
 ?>*\dir{>}\POS?(.5)*+!D{};
 "V0";"V1"**\crv{(0,0)&(10,0)};
 ?>*\dir{>}\POS?(.5)*+!D{};
 "V1";"V2"**\crv{(10,0)&(20,0)};
 ?>*\dir{>}\POS?(.5)*+!D{};
 "V2";"V3"**\crv{(20,0)&(30,0)};
 ?>*\dir{>}\POS?(.5)*+!D{};
 "V3";"V4"**\crv{(30,0)&(40,0)};
 ?>*\dir{>}\POS?(.5)*+!D{};
 (-35,1.5)*+{{e_{-4}}};(-25,1.5)*+{{e_{-3}}};
 (-15,1.5)*+{{e_{-2}}};(-5,1.5)*+{{e_{-1}}};(5,1.5)*+{e_{0}};
 (15,1.5)*+{e_{1}};(25,1.5)*+{{e_{2}}};(35,1.5)*+{{e_{3}}};
 (0.1,-2.5)*+{v_0};(10.1,-2.5)*+{v_1};
 (-9.9,-2.5)*+{v_{-1}};
 (-19.9,-2.5)*+{v_{-2}};
 (-29.9,-2.5)*+{v_{-3}};
 (-39.9,-2.5)*+{v_{-4}}; 
 (20.1,-2.5)*+{v_{2}};
 (30.1,-2.5)*+{v_{3}};
 (40.1,-2.5)*+{v_{4}}; 
\endxy 
\vskip 1.5pc 

\noindent
Given a two-sided  sequence 
$\om=\cdots \om_{-1}\om_0\om_1 \cdots \in \CA^{\mathbb Z}$,  
we obtain a labeled graph  $(E_{\mathbb Z},\CL_\om)$ shown below

\vskip 1.5pc 
\xy  /r0.3pc/:(-44.2,0)*+{\cdots};(44.3,0)*+{\cdots,};
(-40,0)*+{\bullet}="V-4";
(-30,0)*+{\bullet}="V-3";
(-20,0)*+{\bullet}="V-2";
(-10,0)*+{\bullet}="V-1"; (0,0)*+{\bullet}="V0";
(10,0)*+{\bullet}="V1"; (20,0)*+{\bullet}="V2";
(30,0)*+{\bullet}="V3";
(40,0)*+{\bullet}="V4";
 "V-4";"V-3"**\crv{(-40,0)&(-30,0)};
 ?>*\dir{>}\POS?(.5)*+!D{};
 "V-3";"V-2"**\crv{(-30,0)&(-20,0)};
 ?>*\dir{>}\POS?(.5)*+!D{};
 "V-2";"V-1"**\crv{(-20,0)&(-10,0)};
 ?>*\dir{>}\POS?(.5)*+!D{};
 "V-1";"V0"**\crv{(-10,0)&(0,0)};
 ?>*\dir{>}\POS?(.5)*+!D{};
 "V0";"V1"**\crv{(0,0)&(10,0)};
 ?>*\dir{>}\POS?(.5)*+!D{};
 "V1";"V2"**\crv{(10,0)&(20,0)};
 ?>*\dir{>}\POS?(.5)*+!D{};
 "V2";"V3"**\crv{(20,0)&(30,0)};
 ?>*\dir{>}\POS?(.5)*+!D{};
 "V3";"V4"**\crv{(30,0)&(40,0)};
 ?>*\dir{>}\POS?(.5)*+!D{};
 (-35,1.5)*+{\om_{-4}};(-25,1.5)*+{\om_{-3}};
 (-15,1.5)*+{\om_{-2}};(-5,1.5)*+{\om_{-1}};(5,1.5)*+{\om_0};
 (15,1.5)*+{\om_1};(25,1.5)*+{\om_{2}};(35,1.5)*+{\om_{3}};
 (0.1,-2.5)*+{v_0};(10.1,-2.5)*+{v_1};
 (-9.9,-2.5)*+{v_{-1}};
 (-19.9,-2.5)*+{v_{-2}};
 (-29.9,-2.5)*+{v_{-3}};
 (-39.9,-2.5)*+{v_{-4}}; 
 (20.1,-2.5)*+{v_{2}};
 (30.1,-2.5)*+{v_{3}};
 (40.1,-2.5)*+{v_{4}}; 
 (-53,0)*+{(E_{\mathbb Z},\CL_\om)};
\endxy 
\vskip 1.5pc 
\noindent 
where the labeling map $\CL_\om: E_{\mathbb Z}^1\to \CA$ is given  by 
$\CL_\om (e_n)=\om_n$ for $e_n\in E_{\mathbb Z}^1$.
Then we also have a labeled space  $(E_{\mathbb Z},\CL_\om, \bE_{\mathbb Z})$  
with the smallest accommodating set $\bE_{\mathbb Z}$ 
which is closed under relative complements.

\vskip 1pc

\noindent 
{\bf Assumption}. In this section, unless stated otherwise, 
$\CA$ is a finite alphabet with $|\CA|\geq 2$ and 
$\om \in\CA^{\mathbb Z}$ denotes an almost periodic sequence  
such that the  subshift $(\CO_\om, T)$ is a  
Cantor minimal system.

\vskip 1pc 

\subsection{The fixed point algebra 
$C^*(E_{\mathbb Z},\CL_\om, \bE_{\mathbb Z})^\gm$ of the gauge action $\gm$} 
Let $C^*(E_{\mathbb Z},\CL_\om,\bE_{\mathbb Z})=C^*(s_a, p_A)$ be 
the labeled graph $C^*$-algebra 
associated with the labeled space $(E_{\mathbb Z},\CL_\om,\bE_{\mathbb Z})$. 
Since the labeled paths $\CL_\om(E_{\mathbb Z}^{\geq 1})$ 
are exactly the factors of the sequence $\om$, 
from now on we briefly denote 
the whole labeled paths by $L_\om$.

By (\ref{fixed point}), we know that 
the fixed point algebra of the gauge action $\gm$ 
is generated by elements of the form $s_\af p_A s_\bt^*$ ($|\af|=|\bt|$). 
But, in the case $(E_{\mathbb Z},\CL_\om,\bE_{\mathbb Z})$, 
 it is rather obvious that  
$s_\af p_A s_\bt^*\neq 0$, $|\af|=|\bt|$, only if 
 $\af=\bt$ and $A\cap r(\af)\neq \emptyset$. 
 Since $\CL_\om(E^l v)$ consists of 
 a single path for each vertex $v$ and $l\geq 1$,    
 every generalized vertex $[v]_l$ is equal to 
 the range  $r(\af)$ for  a path $\af$ with $\CL_\om(E^l v) =\{\af\}$. 
Hence, by Proposition~\ref{prop-barE},  
we have  
\begin{eqnarray}\label{fixed point} 
C^*(E_{\mathbb Z},\CL_\om,\bE_{\mathbb Z})^\gm 
  =\overline{\rm span}\{s_\af p_{r(\bt\af)} s_\af^*: 
  \af, \bt \in  L_\om \}.
\end{eqnarray} 
Moreover $C^*(E_{\mathbb Z},\CL_\om,\bE_{\mathbb Z})^\gm$ 
is easily seen to be a commutative $C^*$-algebra.
 For each $k\geq 1$, let 
$$F_k:={\rm span}\{ s_\af p_{r(\af'\af)} s_\af^*: \af,\af'\in  L_\om, 
|\af|=|\af'|=k \}.$$ 
The (finitely many) elements $s_\af p_{r(\af'\af)} s_\af^*$ in 
$F_k$ are linearly independent and actually orthogonal to each other so that 
$F_k$ is a finite dimensional subalgebra of $C^*(E_{\mathbb Z},\CL_\om, \bE_{\mathbb Z})^\gm$. 
Moreover  $F_k$ is a subalgebra of $F_{k+1}$ because 
$$s_\af p_{r(\af'\af)} s_\af^*
=\sum_{b\in\CA} s_{\af b} p_{r(\af'\af b)} s_{\af b}^*
=\sum_{a,b\in\CA} s_{\af b} p_{r(a\af'\af b)} s_{\af b}^*.$$ 
This gives rise to an inductive sequence 
$\displaystyle F_1 \xrightarrow{\iota_1} F_2 \xrightarrow{\iota_2} \cdots$  
of finite dimensional $C^*$-algebras, where  
the connecting maps $\iota_k:F_k\to F_{k+1}$ are the inclusions for all $k\geq 1$, 
from which 
we obtain an AF algebra $\varinjlim  F_k$. 

\vskip 1pc

\begin{prop}\label{AF structure} 
For the labeled space $(E_{\mathbb Z},\CL_\om,\bE_{\mathbb Z})$, we have
$$C^*(E_{\mathbb Z},\CL_\om, \bE_{\mathbb Z})^\gm= \varinjlim  F_k.$$
\end{prop}

\begin{proof} Since  
$F_k\subset C^*(E_{\mathbb Z},\CL_\om, \bE_{\mathbb Z})^\gm$ for all $k\geq 1$ 
and $\overline{\cup_k F_k}=\varinjlim  F_k$, 
it is clear that $\varinjlim  F_k\subset C^*(E_{\mathbb Z},\CL_\om, \bE_{\mathbb Z})^\gm$.  
Thus it suffices to know that the algebra $\cup_k F_k$ is dense in  
$C^*(E_{\mathbb Z},\CL_\om, \bE_{\mathbb Z})^\gm$  
and then by (\ref{fixed point}) 
we only need to show that for $y:=s_\af p_{r(\bt\af)} s_\af^*$, 
there is  $k\geq 1$ such that  $y\in F_k$.    
If $|\bt\af|=2|\af|$, then $y\in F_k$ for $k=|\af|$. 
If $|\bt\af|> 2|\af|$, then 
$$y=s_\af p_{r(\bt\af)} s_\af^*=
\sum_{\nu\in \CL_\om(E^{|\bt|-|\af|})} 
s_{\af\nu}p_{r(\bt\af\nu)}s_{\af\nu}^*\in F_k$$ 
for $k=|\bt|$. 
If $|\bt\af|< 2|\af|$, we also have 
$$y=s_\af p_{r(\bt\af)} s_\af^*=
\sum_{\nu\in \CL_\om(E^{|\af|-|\bt|})} s_{\af}p_{r(\nu\bt\af)}s_{\af}^*
\in F_k$$ for $k=|\af|$. 
\end{proof}

\vskip 1pc
 
\begin{prop}\label{AF isomorphism} There is a surjective isomorphism  
\begin{eqnarray}\label{eqn AF iso}
\rho: C^*(E_{\mathbb Z},\CL_\om, \bE_{\mathbb Z})^\gm\to C(\CO_\om)
\end{eqnarray} such that 
 $\rho(s_\af p_{r(\af'\af)} s_\af^*)=\chi_{[\af'.\af]}$ 
 for $s_\af p_{r(\af'\af)} s_\af^*\in F_k$, $k\geq 1$. 
\end{prop}

\begin{proof} Note that for each $k\geq 1$, 
the map $\rho_k: F_k\to C(\CO_\om)$ given by 
$$\rho_k(s_\af p_{r(\af'\af)} s_\af^*)=\chi_{[\af'.\af]}$$ 
is a $*$-homomorphism (we omit the proof) such that 
for $y=s_\af p_{r(\af'\af)} s_\af^*\in F_k$, 
$$\rho_k(y)=\rho_{k+1}(\iota_k(y)),$$  
where $\iota_k:F_k\to F_{k+1}$ is the inclusion map.  
In fact, 
$\iota_k(y)= \sum_{a,b\in\CA} s_{\af b} p_{r(a\af'\af b)} s_{\af b}^*$, 
so that 
$$\rho_{k+1}(\iota_k(y))=
\rho_{k+1}\big( \sum_{a,b\in\CA} s_{\af b} p_{r(a\af'\af b)} s_{\af b}^*\big)
=\sum_{a,b\in\CA} \chi_{[a\af'.\af b]}.$$
But $\sum_{a,b\in\CA} \chi_{[a\af'.\af b]}=\chi_{[\af'.\af]}$ is obvious from 
$\cup_{a,b\in\CA}[a\af'.\af b]=[\af'.\af]$. 
Hence, 
there exists a $*$-homomorphism 
$\rho: \varinjlim  F_k \to C(\CO_\om)$ satisfying 
$\rho(y)=\rho_k(y)$ for all $y\in F_k$, $k\geq 1$. 
Since each $\rho_k$ is injective, so is $\rho$.
 
 Now we show that 
$\rho$ is surjective to complete the proof. 
Let $\chi_{{}_t[\bt]}\in C(\CO_\om)$ for $t\in \mathbb Z$ and 
$\bt\in   L_\om$. 
Assuming $t>0$, we can write 
$\displaystyle \chi_{{}_t[\bt]} =\sum_{\af,\nu} \chi_{ [\af.\nu\bt]}$, where 
the sum is taken over all $\af$, $\nu$ with $|\nu|=t$ and $|\af|=|\nu\bt|$. 
Then for $k:=|\bt|+t$, we have 
$$\chi_{{}_t[\bt]}=\rho_k\big(\,\sum_{\af,\nu} s_\af p_{r(\af\nu\bt)}s_\af^*\big)\in \rho(F_k). 
$$ 
In the case  $t\leq 0$, a similar argument shows 
that $\chi_{{}_t[\bt]}\in \rho(F_k)$ for some $k$.
Thus $\rho$ is surjective since     
${\rm span}\{\chi_{{}_t[\bt]} : t\in \mathbb Z,\,\bt\in  L_\om \}$
is a dense subalgebra of $C(\CO_\om)$.
\end{proof}

\vskip 1pc

\begin{remark}\label{measure} 
It follows from general theory for dynamical systems  that  
the systems $(\CO_\om,T)$ considered in this paper have always 
$T$-invariant ergodic probability measure 
(for example, see \cite[Chapter VIII]{Da}).   
If $m_\om$ is such a $T$-invariant ergodic measure, 
the unital commutative  AF algebra $C(\CO_\om)$ of all 
continuous functions on $\CO_\om$
admits a (tracial) state 
$$
f\mapsto \int_{\CO_\om} f {\rm d} m_\om: C(\CO_\om)\to \mathbb C 
$$
which we also write $m_\om$. 
Since $m_\om$  is $T$-invariant, it easily follows that 
$m_\om (\chi_{{}_t[b]})=m_\om (\chi_{{}_t[b]}\circ T)=m_\om (\chi_{{}_{t+1}[b]})$,
and hence  
\begin{eqnarray}\label{T-invariance}
m_\om (\chi_{{}_t[b]})=m_\om (\chi_{[.b]})
\end{eqnarray} 
holds for all $t\in \mathbb Z$ and  $b\in L_\om$.
\end{remark}

\vskip 1pc

\begin{lem}\label{trace on AF} 
Let 
$\rho: C^*(E_{\mathbb Z},\CL_\om,\bE_{\mathbb Z})^\gm\to C(\CO_\om)$ be 
the isomorphism given in {\rm (\ref{eqn AF iso})}. 
Then a $T$-invariant ergodic measure $m_\om$ on $\CO_\om$  
defines a tracial state
$$\tau_0:=m_\om\circ\rho: C^*(E_{\mathbb Z},\CL_\om,\bE_{\mathbb Z})^\gm\to \mathbb C$$ 
on the fixed point algebra 
$ C^*(E_{\mathbb Z},\CL_\om,\bE_{\mathbb Z})^\gm$ 
such that 
for $\af,  \bt\in  L_\om $, 
$$\tau_0(s_\af p_{r(\bt\af)} s_\af^*)=\tau_0 (p_{r(\bt\af)} ).$$  

\end{lem}

\begin{proof}
Note that $p_{r(\bt\af)} =  \sum_{\nu} s_\nu p_{r(\bt\af\nu)}s_\nu^*$, where 
the sum is taken over the paths $\nu$ with $|\nu|=|\bt\af|$. 
We then have
$$\rho(p_{r(\bt\af)})
= \rho(\, \sum_{|\nu|=|\bt\af|} s_\nu p_{r(\bt\af\nu)}s_\nu^*)
= \sum_{|\nu|=|\bt\af|} \chi_{[\bt\af.\nu]}
= \chi_{\cup_\nu [\bt\af.\nu]}=\chi_{[\bt\af.]}.
$$
Thus $$\tau_0 (p_{r(\bt\af)})= m_\om(\chi_{[\bt\af.]}).$$  
On the other hand, 
if $|\bt\af|> 2|\af|$,  
$s_\af p_{r(\bt\af)} s_\af^*=
\sum_{|\nu|= |\bt|-|\af|}  s_{\af\nu}p_{r(\bt\af\nu)}s_{\af\nu}^*$ 
so that 
$$\tau_0 (s_\af p_{r(\bt\af)} s_\af^*)
=m_\om (\sum_{|\nu|= |\bt|-|\af| } \chi_{[\bt.\af\nu]})
=m_\om(\chi_{[\bt.\af]}).$$ 
But the equality $m_\om(\chi_{[\bt\af.]})=m_\om(\chi_{[\bt.\af]})$ 
follows from the fact that 
$m_\om$ is $T$-invariant (see (\ref{T-invariance})). 
The case where $|\bt\af|\leq 2|\af|$ can be done in a similar way.
\end{proof}

\vskip 1pc 

\begin{lem}\label{lemma for trace} 
The labeled graph $C^*$-algebra $C^*(E_{\mathbb Z},\CL_\om,\bE_{\mathbb Z})$ admits a tracial state 
$$\tau_0\circ\Psi: C^*(E_{\mathbb Z},\CL_\om,\bE_{\mathbb Z})\to \mathbb C,$$ 
where 
$\Psi: C^*(E_{\mathbb Z},\CL_\om,\bE_{\mathbb Z})\to 
C^*(E_{\mathbb Z},\CL_\om,\bE_{\mathbb Z})^\gm$ is the conditional expectation onto 
the fixed point algebra of the gauge action. 
\end{lem}

\begin{proof}
To see that $\tau_0\circ\Psi$ is a trace, 
we claim  
\begin{eqnarray} \label{tracial}
\tau_0(\Psi(XY))=  \tau_0(\Psi(YX)) 
\end{eqnarray}
for  
$X,Y\in {\rm span}\{s_\af p_A s_\bt^* : \af,\bt\in  L_\om,\, 
A\in \bE_{\mathbb Z}, \, A\subset r(\af)\cap r(\bt)\}$. 
   Since the map $\tau_0\circ \Psi$ is linear, 
   we only need to show  (\ref{tracial}) for  
$X=s_\af p_A s_\bt^*$ and $Y=s_\mu p_B s_\nu^*$.
 But also by (\ref{representation}),
 it suffices to consider the case of $|\bt|=|\mu|$, so that 
$XY=\delta_{\bt,\mu} s_\af p_{A\cap B} s_\nu^*$. 
  In this case if $|\af|\neq |\nu|$, then $\Psi(XY)=\Psi(YX)=0$ follows immediately. 
Hence now let $|\af|=|\nu|$. 
If $\af\neq \nu$, it is easy to see that $XY=YX=0$ and (\ref{tracial}) holds. 
 If  $\af=\nu$, then 
 $YX= s_\bt p_{B\cap A} s_\bt^*$ and  $XY=s_\af p_{A\cap B} s_\af^*$, 
 and by Lemma~\ref{trace on AF} we have   
$$\tau_0(\Psi(XY))=\tau_0(XY)=\tau_0(s_\af p_{A\cap B} s_\af^*) 
=\tau_0(p_{A\cap B})= \tau_0(\Psi(YX)).$$ 
 The fact that $\tau_0\circ \Psi$ is a state comes from  
$$ (\tau_0\circ \Psi)(1)=\tau_0 \big(\sum_{a,b\in\CA} s_b p_{r(ab)}s_b^*\big) 
 = m_\om (\sum_{a,b\in\CA } \chi_{[a.b]})= m_\om (\chi_{\CO_\om}) 
 =1.$$
\end{proof}

\vskip 1pc
To prove the simplicity of the labeled graph $C^*$-algebra 
$C^*(E_{\mathbb Z}, \CL_\om, \bE_{\mathbb Z})$, 
we need the following lemma which might be well known in the theory of 
dynamical systems, but we provide a proof here for the reader's 
convenience. 
  
\vskip 1pc

\begin{lem}\label{disagreeable} 
Let $\om\in \CA^{\mathbb Z}$ be a sequence which is 
almost periodic but not periodic. 
Then the labeled space $(E_{\mathbb Z}, \CL_\om, \bE_{\mathbb Z})$ is 
disagreeable.
\end{lem}

\begin{proof}
It is enough to show that the labeled space has no repeatable paths 
(see Remark~\ref{repeatable}).
For this, suppose there is a repeatable path  $\af$. 
We may assume that $\af$ has the smallest length. 
If $\bt\in L_\om$,  by the assumption that $\om$ is almost periodic, 
there exists a $d\geq 1$ such that 
every block $\om_{[t,t+d]}$, $t\in \mathbb Z$,  
has $\bt$ as its factor. 
Thus any path $\af^k\in L_\om$, $k$ large enough, has 
$\bt$ as a factor, so that $\af^k=\mu\bt\nu$ for some 
$\mu, \nu\in L_\om\cup\{\epsilon\}$. 
In other words, every $\bt\in L_\om$ must be of the form  
$\bt=\af''\af^l\af'$ for an initial path $\af'$ and 
terminal path $\af''$ of $\af$ and $l\geq 0$. 

Now we can apply this fact to the paths $\bt=\om_{[0,n]}$, $n\geq 1$, 
to obtain that $\om_{[0,\infty)}$ is of the form $\af''\af^\infty$. 
But then, considering the blocks of the form $\om_{[-n,n]}\in L_\om$ 
($n\to \infty$)
we can easily see that 
$\om=(\af)^\infty\af'. \af''(\af)^\infty$, 
where $\af=\af'\af''$. Thus $\om$ is periodic, which is a contradiction. 
\end{proof}

\vskip 1pc

Since we assume that a Cantor system $(\CO_\om,T)$ is minimal, or 
equivalently $\om$ is almost periodic (and not periodic), 
the labeled space $(E_{\mathbb Z}, \CL_\om, \bE_{\mathbb Z})$ considered in 
this section is 
always disagreeable by Lemma~\ref{disagreeable}.

The following theorem shows that there exist 
simple labeled graph $C^*$-algebras that are not 
stably isomorphic to simple graph $C^*$-algebras.
 
\vskip 1pc 

\begin{thm}\label{main thm 1}
Let $\CA$ be a finite alphabet with $|\CA|\geq 2$, 
and let $\om\in \CA^{\mathbb Z}$ be a sequence such that 
the subshift $(\CO_\om,T)$ is a Cantor minimal system. 
Then the labeled graph $C^*$-algebra
$C^*(E_{\mathbb Z},\CL_\om,\bE_{\mathbb Z})$ is a  
non-AF simple unital $C^*$-algebra with a tracial state  
 $\tau$ 
which satisfies 
$$\tau(s_\af p_{r(\nu\af)}s_\bt^*)
=\tau\circ \Psi(s_\af p_{r(\nu\af)}s_\bt^*)=\delta_{\af,\bt}\tau(p_{r(\nu\af)})$$ 
for labeled paths $\af, \bt, \nu\in \CL_\om (E_{\mathbb Z}^{\geq 1})$.
Moreover if the system  $(\CO_\om, T)$ is uniquely ergodic, 
$C^*(E_{\mathbb Z},\CL_\om,\bE_{\mathbb Z})$ has a unique tracial state.
\end{thm}

\begin{proof} 
For the simplicity of  
$C^*(E_{\mathbb Z},\CL_\om,\bE_{\mathbb Z})=C^*(p_A,s_a)$, 
we show that any nonzero homomorphism  
$\pi:C^*(E_{\mathbb Z},\CL_\om,\bE_{\mathbb Z})\to C^*(q_A, t_a)$ 
onto a $C^*$-algebra  generated by  
$q_A:=\pi(p_A),\ t_a:=\pi(s_a)$ for $ A\in \bE_{\mathbb Z}$, 
$a\in \CA$, is faithful. 
Since the labeled space  
$(E_{\mathbb Z},\CL_\om,\bE_{\mathbb Z})$ is disagreeable 
by Lemma~\ref{disagreeable}, 
we see from \cite[Theorem 5.5]{BP2} that 
$\pi$ is faithful whenever 
$\pi(p_{[v]_l})\neq 0$ for all $v\in E^0$ and $l\geq 1$.
Suppose on the contrary that
$$q_{[v]_m}=\pi(p_{[v]_m})=0$$ 
for some $[v]_m=r(\af)$ with $|\af|=m$. 
Since $\af\in L_\om$ and  $\om$ is almost periodic, 
one finds a $d\geq 1$ such that for all $s\geq 0$, 
$$T^{s+j}\om \in [.\af]$$ 
for some $0\leq j\leq d$. 
This means that if  $\bt\in L_\om$ is a block 
with length $|\bt|\geq d$, it must have $\af$ as a factor.
Thus $\bt$ must be of the form $\bt=\bt'\af\bt''$ 
for some $\bt',\bt''\in \CL_\om^\sharp (E)\,(=L_\om \cup\{\epsilon\})$. 
For $\bt$  with $|\bt|\geq d$  we have $q_{r(\bt)}=0$.  
In fact, 
\begin{align*} 
q_{r(\bt)} & \ =  q_{r(\bt'\af\bt'')}=q_{r(r(\bt'\af),\bt'')}\\ 
           & \ = q_{r(r(\bt'\af),\bt''))}t_{\bt''}^*t_{\bt''}q_{r(r(\bt'\af),\bt'')}\\
           & \ \sim t_{\bt''}q_{r(r(\bt'\af),\bt'')}t_{\bt''}^*\\
           & \ \leq q_{r(\bt'\af)}\leq q_{r(\af)}\\
           & \ = q_{[v]_m}=0.
\end{align*}           
On the other hand, since $\pi$ is a nonzero homomorphism, 
there exists a  
$\dt\in  L_\om$ with $q_{r(\dt)}=\pi(p_{r(\dt)})\neq 0$. 
But then, with an $n>\max\{ |\dt|,d\}$, we have   
$$q_{r(\dt)}=\pi(p_{r(\dt)})
=\pi\big(\sum_{|\dt\mu_i|=n} s_{\mu_i}p_{r(\dt\mu_i)}s_{\mu_i}^* \big)
= \sum_{|\dt\mu_i|=n} t_{\mu_i}q_{r(\dt\mu_i)}t_{\mu_i}^* =0,$$
a contradiction, and 
$C^*(E_{\mathbb Z}, \CL_\om,\bE_{\mathbb Z})$ is simple.
 
  With $\bE_{\mathbb Z}$ in place of $\CB$  in (\ref{K1}) 
it is rather obvious that $\mathcal N=\emptyset$ and $\hat\CB=\hat \CB_J=\bE_{\mathbb Z}$.  
Since 
$\chi_A\in  \ker(1-\Phi)$ if and only if 
$\chi_A=\sum_{a\in \CA} \chi_{r(A,a)}$ (see (\ref{Phi})) which  
actually holds for $A=E_{\mathbb Z}^0$, we have  
$K_1(C^*(E_{\mathbb Z},\CL_\om,\bE_{\mathbb Z}))
= {\rm ker}(1-\Phi)\neq 0$.
Thus $C^*(E_{\mathbb Z},\CL_\om,\bE_{\mathbb Z})$ is not AF. 
(We will see later from 
Theorem~\ref{main thm 2} that 
 $K_1(C^*(E_{\mathbb Z},\CL_\om,\bE_{\mathbb Z}))=\mathbb Z$.)

If $\tau_0$ is the tracial state of 
$C^*(E_{\mathbb Z},\CL_\om,\bE_{\mathbb Z})^\gm$ induced by 
an ergodic measure of $(\CO_\om, T)$, 
the tracial state  
$\tau:=\tau_0\circ \Psi:C^*(E_{\mathbb Z},\CL_\om,\bE_{\mathbb Z})\to \mathbb C$ 
of Lemma~\ref{lemma for trace} satisfies 
\begin{eqnarray}\label{trace condition}
\tau(s_\af p_{r(\nu\af)} s_\bt^*)=\delta_{\af,\bt}\tau(p_{r(\nu\af)})
\end{eqnarray}
for 
$s_\af p_{r(\nu\af)}s_\bt^* \in C^*(E_{\mathbb Z},\CL_\om,\bE_{\mathbb Z})$. 

Now let $(\CO_\om, T)$ be uniquely ergodic and 
again let $\tau_0$ be the tracial state on 
the fixed point algebra 
$C^*(E_{\mathbb Z},\CL_\om,\bE_{\mathbb Z})^\gm$ and 
$\tau:=\tau_0\circ \Psi$ the extension of $\tau_0$ to the 
whole labeled graph $C^*$-algebra as before. 
To show that $\tau$ is the unique tracial state on 
$C^*(E_{\mathbb Z},\CL_\om,\bE_{\mathbb Z})$, 
we claim that  if $\tau'$ is a tracial state on 
$C^*(E_{\mathbb Z},\CL_\om,\bE_{\mathbb Z})$, then $\tau'\circ\Psi=\tau'$ holds, and that 
the state $\tau'\circ \rho^{-1}$ on $C(\CO_\om)$ is $T$-invariant.  
For the first claim, suppose $\tau'\circ\Psi\neq\tau'$. Then 
there exists an element $s_\af p_{r(\af)} s_\bt^*$ ($|\bt|<|\af|$) such that 
$\tau'(s_\af p_{r(\af)} s_\bt^*)\neq 0$. 
Since $\tau'$ is tracial,  we have 
$0\ne \tau'(s_\af p_{r(\af)} s_\bt^*)= \tau'(s_\bt^*s_\af p_{r(\af)})$. 
Thus $\af$ must be of the form $\af=\bt\af'$ for some path $\af'$, and then  
$0\neq \tau'(s_\bt^*s_\af p_{r(\af)})=\tau'(s_{\af'} p_{r(\af)})$. 
Again the tracial property of $\tau'$ gives 
$$0\neq \tau'(s_{\af'} p_{r(\af)})=\tau'( p_{r(\af)} s_{\af'})=
\tau'(s_{\af'} p_{r(\af\af')})=\cdots = \tau'(s_{\af'} p_{(r(\af), (\af')^n)})$$ 
for all $n\geq 1$. 
But this means that the generalized vertex $[v]_l:=r(\af)$, $l=|\af|$, 
is not disagreeable emitting only agreeable paths, 
which is a contradiction to Lemma~\ref{disagreeable}.  
To see that the state $\tau'\circ \rho^{-1}:C(\CO_\om)\to \mathbb C$ is $T$-invariant, 
let $\chi_{{}_t[\bt]}\in C(\CO_\om)$. We assume $t>0$. 
Since $$\rho^{-1}(\chi_{{}_t[\bt]})
=\rho^{-1}\big(\sum_{\substack{\af,\bt\\ |\af|=|\sm\bt|=t+|\bt|}} \chi_{[\af.\sm\bt]}\,\big) 
=\sum_{\substack{\af,\bt\\ |\af|=|\sm\bt|=t+|\bt|}} s_{\sm\bt} p_{r(\af\sm\bt)}s_{\sm\bt}^*,$$ 
we have
$\displaystyle\tau'(\rho^{-1}(\chi_{{}_t[\bt]}))
=\tau'\big( \sum_{\substack{\af,\bt\\ |\af|=|\sm\bt|
=t+|\bt|}} p_{r(\af\sm\bt)}\,\big)=\tau'(p_{r(\bt)}).$  
This implies that 
$$\tau'\circ\rho^{-1}(\chi_{{}_t[\bt]})=\tau'\circ\rho^{-1}(\chi_{{}_{t+1}[\bt]})
=\tau'\circ\rho^{-1}(\chi_{{}_t[\bt]}\circ T),$$ 
which can also be shown for $t\leq 0$ in a similar way. 
Thus $\tau'\circ\rho^{-1}$ is $T$-invariant because the span of functions 
$\chi_{{}_t[\bt]}$ is dense in $C(\CO_\om)$.
\end{proof} 
 
\vskip 1pc 
 
\begin{remarks} 
(1) Simplicity of $C^*(E_{\mathbb Z},\CL_\om,\bE_{\mathbb Z})$ 
can also  be shown by analyzing the path structure of the labeled space 
$(E_{\mathbb Z},\CL_\om,\bE_{\mathbb Z})$. 
For a labeled graph $(E,\CL)$, set 
$$\overline{\CL(E^\infty)}:=\{x\in \CA^{\mathbb N}\mid 
x_{[1,n]}\in \CL(E^n) \ \text{for all } n\geq 1\}.$$ 
Then $\CL(E^\infty)\subset \overline{\CL(E^\infty)}$ is obvious, but 
it is possible to have $\CL(E^\infty)\subsetneq \overline{\CL(E^\infty)}$.     
For example, if  $\om= 0^\infty. 0101^201^301^4\cdots\in \{0,1\}^{\mathbb Z}$, 
then 
the path $1^\infty\in \overline{\CL_\om(E_{\mathbb Z})}$ does not appear 
as an infinite labeled path in $\CL_{\om}(E_{\mathbb Z}^\infty)$.  
 We say that a labeled space $(E,\CL,\bE)$ is {\it strongly cofinal} if 
for each $x\in \overline{\CL(E^\infty)}$ and $[v]_l\in \bE$, there exist 
an $N\geq 1$ and a finite number of paths $\ld_1, \dots, \ld_m\in \CL(E^{\geq 1})$ such that 
$$r(x_{[1,N]})\subset \cup_{i=1}^m r([v]_l,\ld_i).$$ 
This definition of strong cofinality is a modification 
of the definitions given in \cite{BP2, JK} and the proof of \cite[Theorem 6.4]{BP2} 
can be slightly modified to prove that if 
 $(E,\CL,\bE)$ is strongly cofinal and disagreeable, the $C^*$-algebra 
 $C^*(E,\CL,\bE)$ is simple. 
If $\om$ is a sequence satisfying the assumption of this section, 
it is not hard to see that the labeled space 
$(E_{\mathbb Z}, \CL_\om, \bE_{\mathbb Z})$ is 
strongly cofinal.  
Then by Lemma~\ref{disagreeable}, we know that 
$C^*(E_{\mathbb Z},\CL_\om,\bE_{\mathbb Z})$ is simple.

(2)  In case $\om$ is the Thue Morse sequence given in 
Example~\ref{Thue Morse},
one can directly show  that the simple unital $C^*$-algebra 
$C^*(E_{\mathbb Z},\CL_\om,\bE_{\mathbb Z})$ 
admits a unique tracial state. 
Moreover,  its exact values on typical elements of the form $s_\af p_A s_\bt^*$ 
can be obtained explicitly, which will be done in \cite{Ki}.
\end{remarks}
 
\vskip 1pc 

\begin{remark} 
If  $(X,T)$ is a Cantor  minimal system, 
$T$ induces an automorphism $T$ of $C(X)$, 
$$T(f)=f\circ T^{-1},\ \ f\in C(X),$$ 
and 
it is well known that the crossed product $C(X)\times_{T} \mathbb Z$ 
is always simple (for example, see \cite{Da}). 
It is also known \cite{GPS} that 
the crossed products $C(X)\times_{T} \mathbb Z$ are not AF 
because $K_1(C(X)\times_{T} \mathbb Z)=\mathbb Z$.
 But they are all $A\mathbb T$ algebras, hence 
 finite algebras of stable rank one, and have real rank zero by \cite{BBEK}. 
Moreover their isomorphism classes are determined by the ordered $K_0$-groups 
$$(K_0(C(X)\times_{T} \mathbb Z), K_0^+(C(X)\times_{T} \mathbb Z), [1]_0)$$  
together with the distinguished order units $[1]_0$, where $1$ is the unit projection 
of the crossed product. 
\end{remark} 

\vskip 1pc
If a Cantor minimal system 
$(\CO_\om, T)$ is uniquely ergodic,  
the following theorem implies together with Theorem~\ref{main thm 1} 
that the crossed product $C(\CO_\om) \times_{T} \mathbb Z$ 
has a unique tracial state, which is well known for 
uniquely ergodic minimal systems $(X, T)$ of infinite 
spaces $X$ (see \cite[Corollary VIII.3.8]{Da}). 

\vskip 1pc

\begin{thm}\label{main thm 2} 
Let $\CA$ be a finite alphabet with $|\CA|\geq 2$, 
and let $\om\in \CA^{\mathbb Z}$ be a sequence such that 
the subshift $(\CO_\om,T)$ is a Cantor minimal system. 
Then there is an isomorphism   
$$\pi: C^*(E_{\mathbb Z}, \CL_\om,\bE_{\mathbb Z})\to 
C(\CO_\om) \times_{T} \mathbb Z$$ 
such that the restriction of $\pi$ onto 
the fixed point algebra $C^*(E_{\mathbb Z}, \CL_\om,\bE_{\mathbb Z})^\gm$
of the gauge action $\gm$ is an isomorphism onto 
$C(\CO_\om)$.
\end{thm} 

\begin{proof}
Proposition~\ref{AF isomorphism} (and its proof) says that 
the fixed point algebra $C^*(E_{\mathbb Z}, \CL_\om,\bE_{\mathbb Z})^\gm$ is 
isomorphic to $C(\CO_\om)$ via the map $\rho$ 
 given by 
$$\rho(s_\af p_{r(\bt\af)}s_\af^*)=\chi_{[\bt.\af]},\ \ 
   \af,\bt\in \CL_\om^\sharp (E_{\mathbb Z}).$$ 
We show that there exists an isomorphism 
of $C^*(E_{\mathbb Z}, \CL_\om,\bE_{\mathbb Z})$ onto 
the crossed product 
$C^*(E_{\mathbb Z}, \CL_\om,\bE_{\mathbb Z})^\gm\rtimes_{T'} \mathbb Z$,
 where 
 $T':=\rho^{-1}\circ T\circ \rho$ is the automorphism  of 
$C^*(E_{\mathbb Z}, \CL_\om,\bE_{\mathbb Z})^\gm$.  
 
  Note first that  $T'$ satisfies the following 
\begin{eqnarray}\label{T'} 
T'(s_\af p_{r(\bt\af)}s_\af^*)=s_{\af_2\cdots\af_n}p_{r(\bt\af)}s_{\af_2\cdots\af_n}^*
\end{eqnarray} 
for $\af,\bt\in \CL_\om^\sharp (E_{\mathbb Z})$. 
In fact, 
$
\rho\big(T'(s_\af p_{r(\bt\af)}s_\af^*)\big) =T(\rho(s_\af p_{r(\bt\af)}s_\af^*)) 
    = T(\chi_{[\bt.\af]})  
    =\chi_{T([\bt.\af])} 
    =\chi_{[\bt\af_1.\af_2\cdots \af_n]} 
     =\rho\big(s_{\af_2\cdots\af_n}p_{r(\bt\af)}s_{\af_2\cdots\af_n}^*\big) 
$  where $n:=|\af|$.
With the unitary $u$  implementing the automorphism $T'$ 
(namely, $T'=Ad u$), 
this can be written as 
$$ T'(s_\af p_{r(\bt\af)}s_\af^*)
= u(s_\af p_{r(\bt\af)}s_\af^*)u^*
= s_{\af_2\cdots\af_n}p_{r(\bt\af)}s_{\af_2\cdots\af_n}^*.$$ 
Particularly, 
\begin{eqnarray} \label{eqn-t_a}
up_{r(\bt)}u^*=u\big(\sum_{a\in \CA}s_ap_{r(\bt a)} s_a^*\big)u^*= 
\sum_{a\in \CA} p_{r(\bt a)}  
\end{eqnarray}
holds. 
To find a desired isomorphism, we will find a representation of the labeled 
space $(E_{\mathbb Z}, \CL_\om,\bE_{\mathbb Z})$ in the crossed product, and then 
apply the universal property of the $C^*$-algebra 
$C^*(E_{\mathbb Z}, \CL_\om,\bE_{\mathbb Z})$. 
  We actually show that the following partial isometries  
  $$t_a:=u^*p_{r(a)},\ \  a\in \CA$$  
  in the crossed products 
  $C^*(E_{\mathbb Z},\CL_\om,\bE_{\mathbb Z})^\gm\times_{T'} \mathbb Z$ 
  form a representation of $(E_{\mathbb Z},\CL_\om,\bE_{\mathbb Z})$ 
  together with the family of projections $\{p_A : A\in \bE_{\mathbb Z})$.
By (\ref{eqn-t_a}), 
$t_a^* t_a=p_{r(a)}$ and $t_a^* t_b=\delta_{a,b} p_{r(a)}$  
are immediate for  $a,b\in \CA$. 
We also have 
\begin{align*}
 p_{r(\bt)}t_a & =p_{r(\bt)}u^*p_{r(a)}
  = u^*\big(\sum_{b\in \CA} p_{r(\bt b)}\big) p_{r(a)}\\ 
  & = u^* p_{r(\bt a)} =u^* p_{r(a)} p_{r(\bt a)} 
  = t_a p_{r(\bt a)} \\ 
  & =t_a p_{r(r(\bt), a)}. 
\end{align*}      
Since every $A\in \bE_{\mathbb Z}$ can be written as 
a finite union of generalized vertices (by Proposition~\ref{prop-barE}) 
and a generalized vertex $[v]_l$ is clearly equal to 
a range $r(\bt)$ of $\bt\in \CL_\om(E^lv)$,  
we know that  
the above equalities hold for any $A\in\bE_{\mathbb Z}$. 
 Finally we have to check 
$$p_{r(\bt)}= \sum_{a\in \CA} t_a p_{r(\bt a)} t_a^*,$$
 but this follows  directly  from the definition of $t_a$ and (\ref{eqn-t_a}). 
Thus $\{t_a, p_A\}$ forms a representation of the labeled space 
$(E_{\mathbb Z}, \CL_\om, \bE_{\mathbb Z})$ in the $C^*$-algebra
$C^*(E_{\mathbb Z}, \CL_\om,\bE_{\mathbb Z})^\gm\rtimes_{T'} \mathbb Z$, 
and hence 
there exists a homomorphism 
$$\pi:  C^*(E_{\mathbb Z}, \CL_\om,\bE_{\mathbb Z})\to 
C^*(E_{\mathbb Z}, \CL_\om,\bE_{\mathbb Z})^\gm\rtimes_{T'} \mathbb Z$$ 
such that 
$\pi(s_a)=t_a$  and $\pi(p_A)=p_A$ ($a\in \CA$, $A\in \bE_{\mathbb Z}$).  
The homomorphism $\pi$ is injective since   
$C^*(E_{\mathbb Z}, \CL_\om,\bE_{\mathbb Z})$ is simple 
by Theorem~\ref{main thm 1}, 
and is surjective since $u^*=u^*(\sum_{a\in \CA}p_{r(a)})= 
\sum_{a\in \CA} t_a \in  {\rm Im}(\pi)$ and  
$s_{\af}p_{r(\bt\af)}s_{\af}^*=(u^*)^{|\af|}p_{r(\bt\af)}u^{|\af|}\in {\rm Im}(\pi)$ 
for all generators $s_{\af}p_{r(\bt\af)}s_{\af}^*$ of 
$C^*(E_{\mathbb Z}, \CL_\om,\bE_{\mathbb Z})^\gm$.

For the last assertion, it is enough to see that 
for $\af,\bt\in \CL_\om^\sharp (E_\mathbb Z)$,
$$\pi(s_\af p_{r(\bt\af)} s_\af^*)=s_\af p_{r(\bt\af)} s_\af^*.$$  
If $a\in \CA$, then $\pi(p_{r(a)})=\pi(s_a^* s_a)=t_a^* t_a
=p_{r(a)}u u^* p_{r(a)}=p_{r(a)}$, and hence 
$\pi(p_{r(\af)})= p_{r(\af)}$ holds for all $\af\in \CL_\om^\sharp (E_\mathbb Z)$.
The equality (\ref{T'}) shows that the inverse $(T')^{-1}$ of 
the automorphism $T'$ on  $C^*(E_{\mathbb Z}, \CL_\om,\bE_{\mathbb Z})^\gm$ 
maps a projections $p_{r(\af)}$ to the  projection
$s_a p_{r(\af)} s_a^*$,
where 
$a\in \CA$ is the last letter of $\af$. 
(If $\af=\epsilon$  is the empty word, 
$p_{r(\epsilon)}=s_\epsilon$ is the unit of 
the labeled graph $C^*$-algebra, hence 
$(T')^{-1}(p_{r(\epsilon)})=p_{r(\epsilon)}= 
s_{\epsilon}p_{r(\epsilon)}s_\epsilon^*$ also holds.)
Then for $\af=\af'a$ with 
$\af'\in \CL_\om^\sharp (E_\mathbb Z)$, $a\in \CA$, 
we have 
$$
\pi(s_a p_{r(\af)} s_a^*) =t_a p_{r(\af)} t_a^* = u^* p_{r(\af)} u  
 = (T')^{-1} (p_{r(\af)})  = s_a p_{r(\af)} s_a^*
$$
as desired. 
\end{proof}

\end{document}